\documentclass[leqno,sumlimits,intlimits,namelimits]{amsart}

\usepackage{amssymb,latexsym}

\usepackage{a4wide}

\usepackage{eucal}
\usepackage{mathrsfs}

\def\pg{\mathhexbox278}

\newcommand{\Con}{\ensuremath{\mathcal{C}}}
\newcommand{\Conc}{\ensuremath{\mathcal{C}_{\text{c}}}}
\newcommand{\Cinf}{\ensuremath{\mathcal{C}^\infty}}

\newcommand{\D}{\ensuremath{\mathcal{D}}}
\newcommand{\G}{\ensuremath{\mathcal{G}}}
\renewcommand{\S}{\mathscr{S}}

\newcommand{\mb}[1]{\ensuremath{\mathbb{#1}}}
\newcommand{\N}{\mb{N}}

\newcommand{\R}{\mb{R}}
\newcommand{\C}{\mb{C}}

\renewcommand{\d}{\ensuremath{\partial}}

\newfont{\bl}{msbm10 scaled \magstep2}

\newtheorem{theorem}{Theorem}[section]

\newtheorem{proposition}[theorem]{Proposition}

\newtheorem{corollary}[theorem]{Corollary}

\theoremstyle{definition}
\newtheorem{remark}[theorem]{Remark}
\newtheorem{example}[theorem]{Example}

\newcommand{\beq}{\begin{equation}}
\newcommand{\eeq}{\end{equation}}

\newcommand{\dis}[2]{\langle #1 , #2 \rangle}
\newcommand{\inp}[2]{\langle #1 | #2 \rangle}  
\newcommand{\notmid}{\mid\kern-0.5em\not\kern0.5em}

\newcommand{\norm}[2]{{\| #1 \|}_{#2}}

\newcommand{\al}{\alpha}
\newcommand{\be}{\beta}
\newcommand{\ga}{\gamma}

\newcommand{\de}{\delta}
\newcommand{\eps}{\varepsilon}

\newcommand{\vphi}{\varphi}

\newcommand{\la}{\lambda}

\newcommand{\ovl}[1]{\overline{#1}}

\begin{document}

\pagestyle{plain}

\title{The Cauchy problem for Schr\"odinger-type partial differential operators with generalized functions in the principal part and as data}

\author{G\"unther H\"ormann}

\dedicatory{Dedicated to Peter Michor on the occasion of his 60th birthday\\
celebrated at the Central European Seminar in Mikulov, Czech Republic, May 2009}

\address{Fakult\"at f\"ur Mathematik\\
Universit\"at Wien}

\email{guenther.hoermann@univie.ac.at}

\thanks{Research supported by the FWF START grant Y237}


\date{\today}

\begin{abstract}
We set-up and solve the Cauchy problem for Schr\"odinger-type differential operators with generalized functions as coefficients, in particular, allowing for distributional coefficients in the principal part. Equations involving such kind of operators appeared in models of deep earth seismology. We prove existence and uniqueness of Colombeau generalized solutions and analyze the relations with classical and distributional solutions. Furthermore, we provide a construction of generalized initial values that may serve as square roots of arbitrary probability measures. 
\end{abstract}

\maketitle

\section{Introduction}

Partial (and pseudo-) differential operators of
Schr\"odinger-type  with variable coefficients in the principal part arise in the form of so-called \emph{paraxial equations} in models of wave propagation. These are based on narrow-angle  symbol approximations of wave operators and have been applied in integrated optics, underwater acoustic tomography, reflection seismic imaging, time-reversal mirror experiments (cf.\ \cite{BPR:02,Claerbout:85,WMcC:83}), and recently in \cite{dHHO:08} to seismic wave propagation near the core-mantle boundary inside the Earth (at approximately 2800 km depth). 
Since paraxial equations are used to split the wave fields according to a
prescribed principal direction of propagation they are also called \emph{one-way wave equations}. 
The leading-order approximation leads to model equations of Schr\"odinger-type, where the material properties are still reflected by the  regularity structure of the coefficients in the principal part.
Under strong smoothness conditions on the wave speed function (i.e., the coefficients in the original wave operators) well-posedness of the one-way wave Cauchy problems has been discussed in \cite{HT:88,TH:86}. In \cite{dHHO:08} this smoothness
assumption has been considerably relaxed by allowing the coefficients to be of H\"older- or Sobolev-type regularity below log-Lipschitz continuity. Recall that in general existence of distributional solutions to the second-order wave equation may fail below log-Lipschitz regularity of the coefficients (cf.\ \cite{CL:95}). Beyond such coefficient regularity barrier unique solvability of the Cauchy problem holds in the sense of Colombeau generalized functions (cf.\ \cite{GMS:09}), even in a covariant setting.

The relevance of coefficients with low H\"older or Sobolev regularity has been shown in a variety of geophysical applications, e.g.\ in the study of phase transitions in Earth's lowermost mantle  (\cite{science:07}) or in exploration geophysics (\cite{DBR:98,Herrmann:97a,HdH:01c,LH:95,SM:93,SMC:94,Wapenaar:98}).  
The model analyzed in \cite{dHHO:08} describes the paraxial approximation to seismic wave propagation near the core-mantle boundary  by the following Schr\"odinger-type equation with depth $z$ as evolution variable, the $2$-dimensional lateral $x$-variable, and a pseudodifferential time-frequency dependence
\beq \label{PsDO}
  \d_z u - \mathrm{i}\, \big(\d_{x_1} (c_1 \d_{x_1} u) + 
     \d_{x_2} (c_2 \d_{x_2} u)\big) = 0,
\eeq
where $c_1(z,x,\tau)$ and $c_2(z,x,\tau)$ are strictly positive symbols of order $-1$ in $\tau$, continuously differentiable with respect to $z$, but  a non-Lipschitz dependence regarding the $x$-regularity, e.g.\ of  Sobolev type $H^{r+1}$ or H\"older continuous with exponent $r$, where $r < 1$.  Roughly speaking, the main result of \cite{dHHO:08} is the unique solvability of the corresponding Cauchy problem for the wave component $u$ in  $\Con([0,\infty[, \S'(\R;H^2(\R^2))$ (i.e., continuous dependence on $z$, arbitrary $\S'$ quality with respect to time or frequency, but lateral $x$-regularity of class $H^2$), given initial data in $L^1(\R,H^2(\R^2))$ (at $z = 0$) and a right-hand side in $\Con^1([0,\infty[,L^1(\R,L^2(\R^2)))$.
The emphasis in this result is on the  (H\"older or) Sobolev regularity of the solution (with respect to $x$) in relation to the initial data regularity under lowest possible regularity assumptions on the coefficient (cf.\ the brief discussion on inverse analysis of medium regularity at the end of \cite{dHHO:08}).

In the current paper we strive for a widening of the possible range of applications and thereby also simply for a more unified general mathematical set-up. Thus, we allow for discontinuous or distributional coefficients, initial data, and right-hand sides. In fact, a natural method is then to place the whole Cauchy problem into the context of nonlinear theories of generalized functions, in particular the differential algebras of Colombeau generalized functions (cf.\ \cite{Colombeau:83,Colombeau:84,Colombeau:92,GKOS:01,O:92}). 
In view of models from geophysics this amounts to including  discontinuous material properties (e.g.\ at fault zones and geological boundaries or cracks) and Dirac-type data such as strong impulsive sources (e.g.\ explosions or earth quakes).
A quantum field theoretic application of Colombeau generalized functions in terms of regularizations of Wightman distributions has been given in \cite{GOT:00}. 
In the simpler context of quantum mechanics one will of course be interested in allowing for a singular zero order term in the operator to represents a potential and non-smooth data corresponding to an initial probability distribution.  The square of the modulus of a solution to the standard Schr\"odinger equation is usually interpreted as (time evolution of a) probability density. Why should one not turn this around in the sense of allowing now for generalized initial data which represent a ``square root of a given arbitrary probability measure''?\footnote{My thanks go to Gebhard Gr\"ubl and Michael Oberguggenberger who brought up this viewpoint in joint discussions at the University of Innsbruck in April 2002.} 
We will show below how to construct a Colombeau generalized function whose square is associated with a given probability measure in the sense of distributional shadows. The issue of squares of distribution theoretic objects or generalized functions which contain or model measures (in particular, Dirac-type terms) also arises in general relativity theory, e.g.\ in the pseudo-Riemannian metric associated with impulsive pp-waves  (cf.\ \cite[Section 5.3]{GKOS:01} and \cite{KS:99,Steinbauer:98,SV:06}) and in the stress energy tensor corresponding to ultrarelativistic Reissner-Nordstr{\o}m fields (cf.\ \cite{Steinbauer:97}). A different regularization approach for powers of Dirac measures as  initial value appears in \cite{NPR:05} in the context of (semilinear) heat equations with singular potentials. 

In our analysis of the Cauchy problem we will return to denoting the evolution variable by $t$ and furthermore incorporate also first-order terms in the differential operator. Note that unlike in the specific geophysical model situation of \cite{dHHO:08} we neglect additional pseudodifferential aspects in the symbol, since these played only the role of an external parameter upon a partial Fourier transform.
In summary of the basic structure, let $T > 0$ be arbitrary. We consider the Cauchy problem for a generalized function $u$ on $\R^n \times [0,T]$ in the form
\begin{align}
  \d_t u - \mathrm{i}\, \sum_{k=1}^n   \d_{x_k} (c_k 
    \d_{x_k} u) 
    - i V u &= f \label{SCPDE}\\
    u \mid_{t=0} &= g, \label{SCIC} 
\end{align}
where $c_k$ 
 ($k=1,\ldots,n$), 
 $V$, and $f$ are generalized functions on $\R^n \times [0,T]$ and $g$ is a generalized function on $\R^n$.
We note that basics of a theory for abstract variational problems in the context of Colombeau-spaces have been presented in \cite{GV:09}. As indicated in  \cite[Section 8]{GV:09} the Lax-Milgram-type theorem established there provides solutions to Dirichlet-problems involving differential operators similar to the spatial derivatives appearing in the left-hand side of our above differential equation. However, this does not directly extend to our kind of  evolution problem and furthermore involves less strict solution spaces than we are willing to accept here, namely $\G_{H^1}$ as compared to $\G_{H^\infty}$. 
Colombeau-generalized solutions to linear and nonlinear Schr\"odinger equations with constant coefficient principal part have been constructed previously in \cite{Bu:96,Sto:06b,Sto:06a}. 

We may point out that our current paper has both focus and aim different from the enormous literature on the prominent questions concerning  properties of $L^p$-solutions and spectral theory for Schr\"odinger equations with standard  principal part (i.e., with constant or smoooth coefficients) and singular potentials. Instead we investigate here the feasibility of extending the solution concept and the basic analysis to the case of generalized functions and, in particular, non-smooth or highly singular principal parts.
The plan of the paper is as follows. In Section 2 we briefly review the regularization approach to generalized functions in the sense of Colombeau and show how square roots of probability measures can be implemented in this framework. Section 3 presents the main result of unique existence of generalized solutions to our Schr\"odinger-type Cauchy problem (\ref{SCPDE}-\ref{SCIC}). Furthermore, we discuss the relation of Colombeau generalized solutions with classical and distributional solution concepts.


\section{Regularization of coefficients and data}

The basic idea of regularization methods is to replace non-smooth data by approximating nets of smooth functions, e.g.\ instead of $g$ consider $\Cinf \ni g_{\eps} \to g$ ($\eps \to 0$). More generally, we may replace $g$ by a net $(g_{\eps})_{0< \eps \leq 1}$ of $\Cinf$ functions, convergent or not, but with \emph{moderate} asymptotics with respect to $\eps$ and identify regularizing nets whose differences compared to the moderateness scale are \emph{negligible}. For a modern introduction to Colombeau algebras we refer to \cite{GKOS:01}. Here we will also make use of constructions and notations from \cite{Garetto:05b}. 

Construction of generalized functions based on a locally convex topological vector space $E$:

Let $E$ be a locally convex topological vector space whose topology is given by the family of seminorms $\{p_j\}_{j\in J}$. The elements of  
$$
 \mathcal{M}_E := \{(u_\eps)_\eps\in E^{(0,1]}:\, \forall j\in J\,\, \exists N\in\N\quad p_j(u_\eps)=O(\eps^{-N})\, \text{as}\, \eps\to 0\}
$$
and
$$
 \mathcal{N}_E := \{(u_\eps)_\eps\in E^{(0,1]}:\, \forall j\in J\,\, \forall q\in\N\quad p_j(u_\eps)=O(\eps^{q})\, \text{as}\, \eps\to 0\},
$$ 
are called \emph{$E$-moderate} and \emph{$E$-negligible}, respectively. With operations defined componentwise, e.g.\ $(u_\eps) + (v_\eps) := (u_\eps + v_\eps)$ etc., $\mathcal{N}_E$ becomes a vector subspace of $\mathcal{M}_E$.  We define the \emph{generalized functions based on $E$} as the factor space $\G_E := \mathcal{M}_E / \mathcal{N}_E$. If $E$ is a differential algebra then $\mathcal{N}_E$ is an ideal in $\mathcal{M}_E$ and $\G_E$ becomes a differential algebra too.

By particular choices of $E$ we reproduce several standard Colombeau algebras of generalized functions. For example, $E=\C$ with the absolute value as norm yields the generalized complex numbers $\G_E = \widetilde{\C}$ and for any $\Omega \subseteq \R^d$ open, $E=\Cinf(\Omega)$ with the topology of compact uniform convergence of all derivatives provides the so-called special Colombeau algebra $\G_E=\G(\Omega)$. Recall that $\Omega \mapsto \G(\Omega)$ is a fine sheaf, thus, in particular, the restriction $u|_B$ of $u\in\G(\Omega)$ to an arbitrary open subset $B \subseteq \Omega$ is well-defined and yields $u|_B \in \G(B)$. Moreover, we may embed $\D'(\Omega)$ into $\G(\Omega)$ by appropriate localization and convolution regularization. 

In case $E \subseteq \D'(\Omega)$ certain generalized functions can be projected into the space of distributions by taking  weak limits: we say that $u \in \G_E$ is \emph{associated} with $w \in \D'(\Omega)$, if $u_\eps \to w$ in $\D'(\Omega)$ as $\eps \to 0$ holds for any (hence every) representative $(u_\eps)$ of $u$. This fact is also denoted by $u \approx w$. 

In the current paper we will consider open strips of the form $\Omega_T = \R^n \times\, ]0,T[ \subseteq \R^{n+1}$ (with $T > 0$ arbitrary) and employ the spaces $E = H^\infty({\Omega_T}) = \{ h \in \Cinf(\ovl{\Omega_T}) : \d^\al h \in L^2(\Omega_T) \; \forall \al\in \N^{n+1}\}$ with the family of (semi-)norms 
$$
  \norm{h}{H^k} = \Big( \sum_{|\al| \leq k} 
    \norm{\d^\al h}{L^2}^2\Big)^{1/2}
   \quad (k\in \N), 
$$
as well as   
$E = W^{\infty,\infty}({\Omega_T}) = \{ h \in \Cinf(\ovl{\Omega_T}) : \d^\al h \in L^\infty(\Omega_T) \; \forall \al\in \N^{n+1}\}$  with the family of (semi-)norms 
$$
  \norm{h}{W^{k,\infty}} = \max_{|\al| \leq k} \norm{\d^\al h}{L^\infty} \quad (k\in \N). 
$$
(Note that $\Omega_T$ clearly satisfies the strong local Lipschitz property \cite[Chapter IV, 4.6, p.\ 66]{Adams:75} and therefore the Sobolev embedding theorem \cite[Chapter V, Theorem 5.4, Part II, p.\ 98]{Adams:75} implies that every element of $H^\infty(\Omega_T)$ and $W^{\infty,\infty}(\Omega_T)$ belongs to $\Cinf(\ovl{\Omega_T})$.)

To avoid overloaded subscripts we shall make use of the following (abuses of) notation in the sequel
$$
  \G_{L^2}(\R^n \times [0,T]) := \G_{H^\infty({\Omega_T})}
     \quad\text{ and }\quad
   \G_{L^\infty}(\R^n \times [0,T]) := 
   \G_{W^{\infty,\infty}({\Omega_T})}.
$$
For example and in explicit terms we will represent a generalized initial value $g \in \G_{L^2}(\R^n \times [0,T])$ by a net $(g_\eps)$ with the moderateness property
$$
    \forall k \, \exists m: \quad 
    \norm{g_{\eps}}{H^k} = O(\eps^{-m}) \quad (\eps \to 0)
$$
and similarly for the right-hand side and the coefficients. Asymptotically negligible errors in the initial data, e.g.\ by using a representative $(\widetilde{g_{\eps}})$ instead, are expressed by estimates of the form
$$
    \forall k \, \forall p: \quad 
    \norm{g_{\eps} - \widetilde{g_{\eps}}}{H^k} = O(\eps^{p}) 
    \quad (\eps \to 0).
$$
Similar constructions and notations will be used in case of $E = H^\infty(\R^n)$ and $E = W^{\infty,\infty}(\R^n)$. Note that by Young's inequality (\cite[Proposition 8.9.(a)]{Folland:99}) any standard convolution regularization with a scaled mollifier of Schwartz class provides embeddings $L^2 \hookrightarrow \G_{L^2}$ and $L^p \hookrightarrow \G_{L^\infty}$ ($1 \leq p \leq \infty$).  

As an example of a detailed regularization model we construct Colombeau generalized positive square roots of arbitrary probability measures, which can serve as initial values in the Cauchy problem analyzed below.

\begin{proposition}\label{Wurzel} Let $\mu$ be a (Borel) probability measure on $\R^n$. Choose $\rho \in L^1(\R^n) \cap W^{\infty,\infty}(\R^n)$ to be positive with $\int \rho = 1$ and satisfying $\rho(x) \geq |x|^{-m_0}$  when $|x| \geq 1$ with some $m_0 > n$. Set  $\rho_\eps(x) = \frac{1}{\eps^n}\rho(\frac{x}{\eps})$ and $h_\eps := \mu * \rho_\eps$, then we have

\begin{enumerate}

\item   $h_\eps$ is positive and setting $\phi_{\eps} := \sqrt{h_\eps}$ the net $(\phi_\eps)_{\eps\in\,]0,1]}$ represents an element $\phi \in \G(\R^n)$ such that $\phi^2 \approx \mu$;

\item there exists a generalized function $g \in \G_{L^2}(\R^n)$  such that $g^2 \approx \mu$ and  the class of $(g_\eps|_{\Omega})_{\eps\in\,]0,1]}$ is equal to $\phi|_\Omega$ in $\G(\Omega)$, or by slight abuse of notation $g|_\Omega = \phi|_\Omega$, for every bounded open subset $\Omega \subseteq \R^n$.
\end{enumerate}
\end{proposition}

\begin{proof} (1) The Borel measure $\mu$ is regular since it is finite and $\R^n$ is locally compact and second countable (\cite[Proposition 7.2.3]{Cohn:80} or also \cite[Theorem 7.8]{Folland:99}). Thus $\mu(\R^n) = 1$ implies that we can find a compact subset $A \subseteq \R^n$ such that $\mu(A) \geq 1/2$. A variant of Young's inequality for measures (cf.\ \cite[Proposition 8.49]{Folland:99}) applied to $\d^\al h_\eps = \mu * \d^\al \rho_\eps$ directly implies that the net $(h_\eps)$ is $\Cinf(\R^n)$-moderate, even with global $L^\infty$-norms, since $\norm{\mu * \d^\al \rho_\eps}{L^\infty} \leq \norm{\mu}{\text{var}}\, \norm{\d^\al \rho_\eps}{L^\infty} = \norm{\d^\al \rho}{L^\infty} \eps^{-n-|\al|}$. (Here $\norm{\ }{\text{var}}$ denotes the total variation norm on the space of finite Radon measures on $\R^n$, which gives $\norm{\mu}{\text{var}} = \mu(\R^n) = 1$ in case of the positive probability measure $\mu$). 

Furthermore, for any $x \in \R^n$ and $\eps > 0$ we have
$$
  h_\eps (x) = \mu * \rho_\eps(x) = 
  \int_{\R^n} \rho_\eps(x-y) d\mu(y)
  \geq \int_{A} \rho_\eps(x-y) d\mu(y) \geq
   \mu(A) \cdot \!\! \min_{z \in \{x\} - A} \rho_\eps(z) > 0
$$
and hence that $\phi_{\eps} = \sqrt{h_\eps} \in \Cinf(\R^n)$ ($0 < \eps \leq 1$).
Moreover, if $K \subseteq \R^n$ is compact then we have\footnote{We may exclude the case $A = K = \{x_0\}$ without loss of generality} with $r(K) := \max \{ |x|: x \in K - A \} > 0$ the estimate
\begin{multline} \tag{$*$}
\inf_{x \in K} h_\eps(x) =  
\inf_{x \in K} \mu * \rho_\eps(x) \geq 
     \mu(A) \cdot \inf_{x \in K}
     \min_{z \in \{x\} - A} \rho_\eps(z)
    \geq  \frac{1}{2} \cdot \min_{z \in K - A} 
      \frac{\rho(z/\eps)}{\eps^n} \\
    \geq \frac{1}{2 \eps^n} \min_{|z| \leq r(K)} \rho(z/\eps)
    \geq \frac{1}{2 \eps^n} \frac{\eps^{m_0}}{r(K)^{m_0}}
    = \frac{\eps^{m_0 - n}}{2\, r(K)^{m_0}}
    \qquad (0 < \eps < 1/r(K)).
\end{multline}
If $\al \in \N^n$ is arbitrary then $\d^\al \phi_\eps$ is a linear combination of terms of the form
$ {\d^{\be_1}h_\eps \cdots\d^{\be_k} h_\eps} /{h_\eps^{l/2}}$
with appropriate $\be_1,\ldots,\be_k \in \N^n$ and $l \in \N$. Hence $\Cinf(\R^n)$-moderateness of $(h_\eps)$ together with the lower bounds for $\inf_{x \in K} h_\eps(x)$ obtained above prove $\Cinf(\R^n)$-moderateness of $\phi_\eps$. 

Finally, we have that by construction $\phi^2$ is represented by $h_\eps = \mu * \rho_\eps$ and thus clearly converges to $\mu$ as $\eps \to 0$ in the sense of distributions.

(2) For $j\in\N$ let $K_j$ be the closed ball of radius $2^j$ around $0$ in $\R^n$. Let $\chi_0 \in \D(\R^n)$, $0 \leq \chi_0 \leq 1$ with $\chi_0(x) = 1$ when $|x| \leq 1$ and $\chi_0(x) = 0$ when $|x| > 2$ and put $\chi_j(x) := \chi_0(2^{-j} x)$ ($j \geq 1$). Thus we have $\chi_j(x) = 1$ when $|x| \leq 2^j$, $\chi_j(x) = 0$ when $|x| > 2^{j+1}$, and for every $\ga\in\N^n$
$$
  \norm{\d^\ga \chi_j}{L^2} = 
    2^{(- |\ga| + n/2) j} \norm{\d^\ga \chi_0}{L^2}.
$$

We define the net $(g_\eps)_{\eps\in\,]0,1]}$ of smooth functions on $\R^n$ by
$$
   g_\eps := \chi_j\, \phi_\eps, \text{ if }\; 2^{-j-1} < \eps \leq 2^{-j}
    \quad (j \in \N).
$$
For every $\eps \in\, ]0,1]$ the function $g_\eps$ is smooth and has compact support, hence belongs to $H^\infty(\R^n)$. We have to show that $(g_\eps)$ is an $H^\infty$-moderate net. 

Let $\al \in \N^n$ arbitrary and $2^{-j-1} < \eps \leq 2^{-j}$. By the Leibniz rule we have that $\d^\al g_\eps$ is a linear combination of terms of the form $\d^\beta \phi_\eps \cdot \d^{\al - \beta} \chi_j$ ($\be \leq \al$). Therefore we have by the triangle inequality that $\norm{\d^\al g_\eps}{L^2}$ is bounded above by a linear combination of terms of the form 
\begin{multline*}
   \norm{\d^\beta \phi_\eps \cdot \d^{\al - \beta} \chi_j}{L^2}
   \leq \norm{\d^{\al - \beta} \chi_j}{L^2} 
   \norm{\d^\beta \phi_\eps}{L^\infty(K_{j+1})}\\ \leq
   2^{nj/2} \, \norm{\chi_0}{H^{|\al|}} 
     \norm{\d^\beta \phi_\eps}{L^\infty(K_{j+1})}
   \leq \eps^{-n/2}\, \norm{\chi_0}{H^{|\al|}} 
      \norm{\d^\beta \phi_\eps}{L^\infty(K_{j+1})}
\end{multline*}
and it remains to prove $\eps$-moderate bounds for $\norm{\d^\beta \phi_\eps}{L^\infty(K_{j+1})}$ when $\beta \leq \al$ (note that the coupling of $j$ with $\eps$ has to be taken into account here). As already noted in (1) the latter is in turn bounded above by a linear combination of terms of the form
$$
    \sup_{x \in K_{j+1}} \frac{|\d^{\be_1}h_\eps(x)| \cdots
       |\d^{\be_k} h_\eps(x)|} {|h_\eps(x)|^{l/2}}
$$
with suitable $\be_1,\ldots,\be_k \in \N^n$ and $l \in \N$. Thanks to the estimate ($*$) in the proof of (1) we have  $1/|h_\eps(x)|^{l/2} \leq 2\, r(K_{j+1})^{m_0} \eps^{-(m_0 - n)}$ uniformly with respect to $x \in K_{j+1}$, where $r(K_{j+1}) = O(2^{j+1}) = O(1/\eps)$. Thus it remains to consider the factors $|\d^{\be_q}h_\eps(x)|$ in the above supremum. As has already been noted in the first paragraph of the proof of (1) we have global $L^\infty$-estimates with moderate $\eps$-dependence. Thus in summary, we have shown the $H^\infty$-moderateness of $(g_\eps)$.  

Finally, let $\Omega \subseteq \R^n$ be bounded and open. Choose $j$ sufficiently large so that $\Omega \subseteq K_j$. Then we have for $0 < \eps < 2^{-j}$ by construction of $g_\eps$ that $g_\eps|_\Omega = \phi_\eps|_\Omega$ and hence equality of the corresponding classes in $\G(\Omega)$. Moreover, the latter also implies that $g^2 \approx \mu$, since the support of any test function in $\D(\R^n)$ is contained in some open bounded subset. 
\end{proof}

\begin{remark}\label{Wurzel_rem} (i) Property ($*$) established in course of the proof shows in fact that the class of $(h_\eps)$ in $\G(\R^n)$ as well as $\phi$ are strictly positive generalized functions (cf.\ \cite[Theorem 3.4]{Mayerhofer:08}). In this sense, $\phi$ provides a strictly positive square root of the positive measure $\mu$.

(ii) For specific choices of $\rho$ in $L^1(\R^n) \cap H^\infty(\R^n)$ such that  $\sqrt{\rho} \in H^\infty(\R^n)$ we could obtain that $(\phi_\eps)$ is also $H^\infty$-moderate and directly defines a square root in $\G_{L^2}(\R^n)$ without having to undergo the cut-off procedure in part (2) of Proposition \ref{Wurzel} (which, on the other hand, cannot be avoided for general $\rho \in H^\infty$). For example, putting $\rho(x) = c (1+|x|^2)^{-(n+1)/2}$ with a suitable normalization constant $c > 0$ provides such a mollifier. However, the  above Proposition leaves considerably more flexibility in adapting the regularization to particular applications.
\end{remark}


\section{Generalized function solutions and coherence properties}

We come now to the main existence and uniqueness result for generalized solutions to the Cauchy problem (\ref{SCPDE}-\ref{SCIC}). We recall that a regularization of an arbitrary finite-order distribution which meets the log-type conditions on the coefficients $c_k$ and $V$ in the following statement is easily achieved by employing a re-scaled mollification process as described in \cite{O:89}.

\begin{theorem}\label{exunthm} Let $c_k$ 
 ($k=1,\ldots,n$) as well as $V$ be real\footnote{In the sense that they possess representating nets of real-valued functions.} generalized functions in $\G_{L^{\infty}}(\R^n \times [0,T])$, $f$ in  $\G_{L^2}(\R^n \times [0,T])$, and $g$ be in $\G_{L^2}(\R^n)$. Assume  that the following \emph{log-type conditions} hold for some (hence all) representatives $(c_{k \eps})_{\eps \in ]0,1]}$ of $c_k$ ($k=1\ldots,n$) and $(V_\eps)$ of $V$:
$$
  \norm{\d_t {c_{k \eps}}}{L^\infty} = O(\log(\frac{1}{\eps}))
  \quad \text{and} 
  \quad \norm{\d_t {V_{\eps}}}{L^\infty} = O(\log(\frac{1}{\eps}))
  \quad (\eps \to 0).
$$
In addition let the positivity conditions $c_{k \eps}(x,t) \geq c_0$ for all $(x,t) \in \R^n \times [0,T]$, $\eps\in\,]0,1]$, $k=1,\ldots,n$ with some constant $c_0 > 0$ be met.\footnote{For any other representative  the conditions thus hold for sufficiently small $\eps$ with $c_0 / 2$ instead.}

Then the Cauchy problem 
\begin{align}
  \d_t u - \mathrm{i}\, \sum_{k=1}^n   \d_{x_k} (c_k 
    \d_{x_k} u) 
    - i V u &= f, \label{PDE}\\
    u \mid_{t=0} &= g, \label{IC}
\end{align}
has a unique solution $u \in \G_{L^2}(\R^n \times [0,T])$. 
\end{theorem}

\begin{proof} In terms of representatives (\ref{PDE}-\ref{IC}) means that we have a family of Cauchy problems  parametrized by $\eps \in \, ]0,1]$: 
\begin{align}
  \d_t u_\eps - \mathrm{i}\, \sum_{k=1}^n   \d_{x_k} (c_{k\eps} 
    \d_{x_k} u_\eps) 
     - i V_\eps u_\eps &= f_\eps \quad \text{ on }†
       \Omega_T = \R^n \times\, ]0,T[, \label{PDE_eps}\\
    u_\eps \mid_{t=0} &= g_\eps \quad \text{ on } \R^n, \label{IC_eps}
\end{align}
where $c_{k\eps}$ ($k=1,\ldots,n$) 
and $V_\eps$ (are real-valued and) belong to $W^{\infty,\infty}(\Omega_T)$, $c_{k \eps}$ satisfies the log-type and the positivity condition as stated above, $g_\eps \in H^\infty(\R^n)$, and $f_\eps \in H^\infty(\Omega_T)$.

Our strategy is to solve the corresponding problem at fixed, but arbitrary,  parameter value $\eps$ and thereby produce a solution candidate $(u_\eps)$. The substance of the proof lies in the efforts to show that each $u_\eps$ belongs to $H^\infty(\R^n \times ]0,T[)$ and in addition satisfies moderateness estimates in every derivative with respect to the $L^2$-norm. We will obtain basic energy estimates by standard variational methods as discussed by Dautray-Lions in \cite[Chapter XVIII, \pg 7, Section 1]{DL:V5}, but we will need to perform an additional analysis of the dependence of the constants in these estimates on the various norms of the coefficient functions.

\emph{Step 1 (basic estimates and regularity):} We apply the set-up and constructions in \cite[Chapter XVIII, \pg 7, Section 1]{DL:V5} to the Hilbert space triple $H^1(\R^n) \subset L^2(\R^n) \subset H^{-1}(\R^n)$ and the sesquilinear form
$$
   a_\eps(t;\vphi,\psi) := 
   \sum_{k=1}^n \inp{c_{k \eps}(t)\, \d_{x_k} 
   \vphi}{\d_{x_k} \psi}_{L^2(\R^n)} + \inp{V_\eps(t) \vphi}{\psi}
   \qquad (\vphi, \psi \in H^1(\R^n)),
$$
where we have used the short-hand notation $v(t)$ for a function $x \mapsto v(x,t)$. 
The basic conditions  \cite[(7.1) and (7.2), p.\ 621]{DL:V5}  on the quadratic form $a_\eps$  are easily seen to be met: continuous differentiability of the map $t \mapsto a_\eps(t;\vphi,\psi)$ follows from smoothness of $c_{k \eps}$ ($k=1,\ldots,n$) and $V_\eps$, the form  $a_\eps(t;.,.)$ is hermitian since each $c_{k \eps}$ and $V_\eps$ is real-valued, and the positivity conditions on the coefficients $c_{k \eps}$  ($k=1,\ldots,n$) in the statement of  Theorem \ref{exunthm} yield the following coercivity estimate
$$
        a_\eps(t;\vphi,\vphi) + \la_\eps \norm{\vphi}{L^2}^2 \geq c_0 \norm{\vphi}{H^1}^2  \qquad \forall \vphi \in H^1(\R^n), 
$$
where $\la_\eps = c_0 + \norm{V_\eps}{L^\infty}$
(thus we have $\lambda = \lambda_\eps$ and $\alpha = c_0$ in  \cite[(7.2), ii) on p.\ 621]{DL:V5} ). Furthermore, the hypotheses \cite[(7.5) and (7.6), pp.\ 621-622]{DL:V5} on the initial values hold by our assumptions on $g_\eps$ and $f_\eps$. We claim that \cite[Theorem 1, pp.\ 621-622]{DL:V5} in combination with the remark on additional regularity \cite[Remark 3, p.\ 625]{DL:V5} implies that we obtain a unique solution 
\beq \tag{$S_1$}
   u_\eps \in \Con([0,T],H^1(\R^n)) \cap \Con^1([0,T],H^{-1}(\R^n))
\eeq 
to the Cauchy problem (\ref{PDE_eps}-\ref{IC_eps}) for every $\eps \in \, ]0,1]$. In fact, the basic theorem gives existence and uniqueness with $u_\eps \in L^2([0,T],H^1(\R^n))$ and $u_\eps' \in L^2([0,T],H^{-1}(\R^n))$ only. However, as mentioned in \cite[Remark 3, p.\ 625]{DL:V5} the existence proof given in Dautray-Lions directly shows that the solution as well as its $t$-derivative are $L^\infty$-functions of $t$ with values in $H^1$ or $H^{-1}$, respectively.  The same remark states that even continuity with respect to $t$ holds (an earlier reference for this fact is \cite[Remark 10.2, p.\ 302]{LM:72}, which also indicates a proof based on convolution regularization). Finally, we note that thanks to $\Con([0,T],H^1) \subseteq  \Con([0,T],H^{-1})$ an application of \cite[Chapter XVIII, \pg 1, Section 2, Proposition 7, p.\ 477]{DL:V5} shows that the solution belongs to $\Con^1([0,T],H^{-1})$.

Moreover, a careful inspection of all constants appearing in the derivation of the a priori estimate \cite[(7.16) and (7.17) on p.\ 624]{DL:V5} shows that we may deduce the following basic energy estimate uniformly for all $t \in [0,T]$
\beq \label{energy}
  \norm{u_\eps(t)}{H^1}^2 \leq 
     C_{2 \eps} \, e^{C_{1 \eps}}
     \left( \norm{g_{\eps}}{H^1}^2  + 
     \int_0^T \left(\norm{f_\eps(\tau)}{L^2}^2 + 
       \norm{\d_t f_\eps(\tau)}{H^{-1}}^2 \right) d\tau
     \right),
\eeq
where the specifics on the constants $C_{1 \eps}$ and $C_{2 \eps}$ according to the derivation in \cite[pp. 623-624]{DL:V5} are as follows:   $C_{2\eps} = O(T (c_0 + \norm{V_\eps}{L^\infty})) = O(\eps^{-p})$ ($\eps \to 0$) for some $p \in \N$, and  
$$
    C_{1 \eps} = O\Big( \frac{T}{c_0} \cdot
    (\max\limits_{1\leq k \leq n} 
       \norm{\d_t c_{k \eps}}{L^\infty}
       + \norm{\d_t V_{\eps}}{L^\infty}) \Big) 
       \quad (\eps \to 0),
$$
which by the log-type conditions on $(c_{k \eps})$ and $(V_\eps)$ imply moderateness estimates for the first-order spatial derivatives in the form
$$
  \exists N_1 \in \N: \quad
   \norm{\d_{x_l} u_\eps}{L^2(\R^n\times[0,T])} \leq
   \sqrt{T} \sup_{t \in [0,T]} \norm{u_\eps(t)}{H^1} = O(\eps^{-N_1}) 
   \qquad (\eps \to 0, \, l=1,\ldots,n).
$$
As a preparation for the procedure in Step 2 we claim that additional regularity properties for $u_\eps$ hold in the form
\beq \tag{$*$}
  u_\eps \in L^\infty([0,T],H^2(\R^n)) 
    \cap W^{1,\infty}([0,T],H^1(\R^n))
\eeq
(at fixed $\eps$ without precise asymptotic estimates of the norms).  These can be obtained from the concept of mild solutions and evolution systems on $L^2(\R^n)$ (cf.\ \cite[Chapter 5]{Pazy:83}) for the self-adjoint familiy of operators $(A_\eps(t))_{t \in [0,T]}$, where $A_\eps(t) v (x):= \sum_{k=1}^n  \d_{x_k} (c_{k\eps}(x,t) \d_{x_k} v)(x)$ ($v \in H^2(\R^n)$)  with common (i.e., $t$-independent) domain $H^2(\R^n)$: In fact, Equation \eqref{PDE} has the form
$$
   \d_t u_\eps - i A_\eps(t) u_\eps = 
      f_\eps 
            + i V_\eps u_\eps =: F_\eps,
$$ 
where we already know from ($S_1$) that $F_\eps \in \Con([0,T],L^2(\R^n)) \cap \Con^1([0,T],H^{-2}(\R^n))$. Let the strongly continuous evolution system corresponding to $(A_\eps(t))_{t \in [0,T]}$ be denoted by $(U_\eps(t,s))_{0 \leq s \leq t \leq T}$, then we  necessarily have by Duhamel's formula 
$$
  u_\eps(t) = U_\eps(t,0) g_\eps + \int_0^t U_\eps(t,\tau) 
    F_\eps(\tau)\, d\tau
    \in \Con_w^1([0,T],L^2(\R^n)), 
$$
where $\Con_w^1([0,T],L^2(\R^n))$ means weakly continuously differentiable as $t$-dependent distribution on $\R^n$ with values in $L^2(\R^n)$ for every $t$.
Writing the differential equation now in the form
$$
   - i A_\eps(t) u_\eps = - \d_t u_\eps + F_\eps \in 
   \Con_w([0,T],L^2(\R^n))
$$ 
we may employ elliptic regularity (spatial, with respect to $x$) for the second-order operator $A_\eps(t)$ and deduce $u_\eps \in \Con_w([0,T],H^2(\R^n)) \subseteq L^\infty([0,T],H^2(\R^n))$. 
Moreover, we may now state in addition that $F_\eps \in \Con_w([0,T],H^{1}(\R^n))$ and use Duhamel's formula again to show that also $u_\eps \in \Con_w^1([0,T],H^1(\R^n)) \subseteq W^{1,\infty}([0,T],H^1(\R^n))$ as claimed.

\emph{Step 2 (higher $x$-derivatives):} We take the partial $x_j$-derivative on both sides in Equations \eqref{PDE_eps} and \eqref{IC_eps} to obtain a similar differential equation for $\d_{x_j} u_\eps$ in the form
$$
  \d_t \d_{x_j} u_\eps - \mathrm{i}\, \sum_{k=1}^n   
       \d_{x_k} (c_{k\eps} \d_{x_k} \d_{x_j} u_\eps)
  - i V_\eps \d_{x_j} u_\eps
  =  \mathrm{i}\, \sum_{k=1}^n   \d_{x_k} (\d_{x_j} c_{k\eps} 
    \d_{x_k} u_\eps) 
    + i \d_{x_j} V_\eps u_\eps
     + \d_{x_j} f_\eps =: f_{1 \eps}
$$
and the initial condition $\d_{x_j} u_\eps (0) = \d_{x_j} g_\eps$. Thanks to ($*$) we have $f_{1 \eps} \in L^2([0,T],L^2(\R^n))$ and $\d_t f_{1 \eps} \in L^2([0,T],H^{-1}(\R^n))$, which corresponds to conditions (7.5-6) of \cite[Chapter XVIII, \pg 7, Section 1, Theorem 1]{DL:V5}. Thus we may apply ($S_1$) and \eqref{energy} with $\d_{x_j} u_\eps$, $\d_{x_j} g_\eps$, $f_{1\eps}$ replacing $u_\eps$, $g_\eps$, $f_\eps$, respectively, but with the same constants $C_1$ and $C_2$. Collecting the results for $j=1,\ldots,n$ we arrive at
\beq \tag{$S_2$}
   u_\eps \in \Con([0,T],H^2(\R^n)) \cap \Con^1([0,T],L^2(\R^n))
\eeq 
and 
$$
  \exists N_2 \in \N: \quad
   \norm{\d_{x_l} \d_{x_j} u_\eps}{L^2(\R^n\times[0,T])} \leq
   n \sqrt{T} \sup_{t \in [0,T]} \norm{u_\eps(t)}{H^2} = O(\eps^{-N_2}) 
   \qquad (\eps \to 0, \, j,l=1,\ldots,n).
$$
Similarly, taking higher partial $x$-derivatives in Equations \eqref{PDE_eps} and \eqref{IC_eps} produces the same kind of differential equations and initial conditions for these higher derivatives, where we may always consider the already estimated lower order derivatives of $u_\eps$ as part of the right-hand side. Therefore we obtain successively for arbitrary $m\in\N$
\beq \tag{$S_m$}
   u_\eps \in \Con([0,T],H^m(\R^n)) \cap \Con^1([0,T],H^{m-1}(\R^n))
\eeq 
and 
$$
  \exists N_m \in \N \,\forall \al \in \N_0^n, |\al| \leq m: \quad
   \norm{\d_x^\al u_\eps}{L^2(\R^n\times[0,T])} = O(\eps^{-N_m}) 
   \qquad (\eps \to 0).
$$

\emph{Step 3 ($t$-derivatives):} We first note that due to ($S_m$) we have $\d_t u_\eps(0) \in H^\infty(\R^n)$ and Equation \eqref{PDE_eps} evaluated at $t=0$ provides moderate $\eps$-asymptotics of its Sobolev norms. Taking the partial $t$-derivative in Equation \eqref{PDE_eps} now yields again the same kind of differential equation for $\d_t u_\eps$ as above when interpreted in the following way:
$$
  \d_t \d_{t} u_\eps - \mathrm{i}\, \sum_{k=1}^n   
       \d_{x_k} (c_{k\eps} \d_{x_k} \d_{t} u_\eps)
  - i V_\eps \d_{t} u_\eps
  =  \mathrm{i}\, \sum_{k=1}^n   \d_{x_k} (\d_{t} c_{k\eps} 
    \d_{x_k} u_\eps) 
    + i \d_{t} V_\eps u_\eps
     + \d_{t} f_\eps =: \widetilde{f}_{1 \eps}.
$$
By property ($S_m$) (now replacing ($*$) in a similar argument in Step 2) we deduce the required regularity conditions on $\widetilde{f}_{1 \eps}$ and $\d_t \widetilde{f}_{1 \eps} \in L^2([0,T],H^{-1}(\R^n))$ to apply ($S_1$) and \eqref{energy} now with $\d_t u_\eps$, $\d_t u_\eps(0)$, $\widetilde{f}_{1\eps}$ replacing $u_\eps$, $g_\eps$, $f_\eps$, respectively. We obtain for $\al \in \N_0^n$, $|\al| \leq 1$,
$$
   u_\eps \in \Con^1([0,T],H^1(\R^n)) \cap \Con^2([0,T],H^{-1}(\R^n)),
  \quad 
  \exists M_1 \in \N: \, 
   \norm{\d_t \d_{x}^\al u_\eps}{L^2(\R^n \times [0,T])} = 
   O(\eps^{-M_1}).
$$
Furthermore, inserting here the same procedure as in Step 1 now yields for any $m \in \N$ and for all $\al \in \N_0^n$, $|\al| \leq m$,
$$
   u_\eps \in \Con^1([0,T],H^m(\R^n)) \cap 
   \Con^2([0,T],H^{m-1}(\R^n)),
  \quad 
  \exists M_m \in \N: \, 
   \norm{\d_t \d_x^\al u_\eps}{L^2(\R^n \times [0,T])} 
   = O(\eps^{-M_m}).
$$
Thus in turn we may now deduce from the above differential equation for $\d_t u_\eps$ that $\d_t^2 u_\eps (0) \in H^\infty(\R^m)$ holds with moderate spatial Sobolev norms. Hence we may now play the same game again with $\d_t^2 u_\eps$ etc.\ and finally arrive at the statements that for arbitrary $d,m \in \N$ and for all $\al \in \N_0^n$, $|\al| \leq m$, we have
$$
   u_\eps \in \Con^d([0,T],H^m(\R^n)) \cap 
   \Con^{d+2}([0,T],H^{m-1}(\R^n)),
  \; 
  \exists M_{md} \in \N: \, 
   \norm{\d_t^d \d_x^\al u_\eps}{L^2(\R^n \times [0,T])} = 
   O(\eps^{-M_{md}}).
$$
Therefore we may conclude that $(u_\eps)$ is moderate net and its class  $\G_{L^2}(\R^n \times [0,T])$ defines a
solution $u$ to the Cauchy problem (\ref{PDE}-\ref{IC}).

The proof of uniqueness requires to show negligibility of any solution $(u_\eps)$ assuming that the right-hand side $(f_\eps)$ and the initial data $(g_\eps)$ are negligible. But this follows successively from the corresponding variant of the energy estimate \eqref{energy} applied at each step of the sequence of differentiated differential equations used in course of the existence proof.
\end{proof}


We note that in case of smooth coefficients an easy integration by parts argument shows that any solution to the Cauchy problem obtained from the variational method as in \cite[Chapter XVIII, \pg 7, Section 1]{DL:V5})
also provides a solution in the sense of distributions. The following statement ensures in addition the coherence with the Colombeau generalized solution.

\begin{corollary}\label{cor} Let $V$ 
and $c_k$ ($k=1,\ldots,n$) belong to $\Cinf(\Omega_T) \cap L^\infty(\Omega_T)$ with bounded time derivatives of first-order,  $g_0 \in H^1(\R^n)$, and $f_0 \in \Con^1([0,T],L^2(\R^n))$. Let $u$ denote the unique Colombeau generalized solution to the Cauchy problem (\ref{PDE}-\ref{IC}), where $g$, $f$ denote standard embeddings of $g_0$, $f_0$, respectively. Then $u \approx w$, where $w \in \Con([0,T],H^1(\R^n))$ is the unique distributional solution obtained from the variational method. 
\end{corollary}

\begin{proof} Let $(u_\eps) \in u$ be as in the proof of Theorem \ref{exunthm}, then $v_\eps := u_\eps - w$ satisfies the following Cauchy problem
$$
  \d_t v_\eps - \mathrm{i}\, \sum_{k=1}^n   \d_{x_k} (c_{k} 
    \d_{x_k} v_\eps) 
    - i V v_\eps = f_\eps - f_0, \qquad
    v_\eps \mid_{t=0} = g_\eps - g_0.
$$
Moreover, we have the energy estimate \eqref{energy}, where now both $C_{1 \eps} =: C_1$ and $C_{2 \eps} =: C_2$ both are independent of $\eps$, in the form
$$
  \sup_{t \in [0,T]} \norm{v_\eps(t)}{H^1}^2 \leq 
     C_2 e^{C_1}
     \Big( \norm{g_{\eps} - g_0}{H^1}^2 +  
     \int_0^T \left(\norm{f_\eps(\tau) - f_0(\tau)}{L^2}^2  
      + 
       \norm{\d_t f_\eps(\tau) - \d_t f_0(\tau)}{H^{-1}}^2 
       \right) d\tau \Big).
$$
Since the smoothing process via mollification ensures convergence to $0$ of the norms on the right-hand side of the above estimate we obtain even norm convergence of $u_\eps$ to $w$. 
\end{proof}

\begin{remark}
To compare our main result with the evolution systems solution constructed in \cite{dHHO:08} for the case of Sobolev coefficients and  spatial dimension $n=2$ we drop the additional pseudodifferential aspects in that model. Then we obtain also coherence of the Colombeau concept with the functional analytic solution as we sketch in the following discussion:  

Let $0 < r < 1$ and $V = 0$, $c_{10}, c_{20} \in \Con^1([0,T],H^{r+1}(\R^2))$, $g_0 \in H^2(\R^2)$, $f_0 \in \Con^1([0,T],L^2(\R^2))$. Let $v$ denote the unique $H^2$-valued solution to the Cauchy problem corresponding to (\ref{PDE}-\ref{IC}) with initial value $g_0$ and right-hand side $f_0$ (and coefficients $c_{j0}$ replacing $c_j$) according to \cite[Theorem 4.2]{dHHO:08}. Let $g$, $f$, $c_1$, $c_2$ denote the embeddings of $g_0$, $f_0$, $c_{10}$, $c_{20}$ into corresponding Colombeau spaces, respectively, and let $u \in \G_{L^2}(\R^2 \times [0,T])$ be the unique Colombeau solution to (\ref{PDE}-\ref{IC}). Subtracting the equations satisfied by a representative $(u_\eps)$ of $u$ and $v$ we deduce the following equations for the difference $h_\eps := u_\eps - v$
$$
  \d_t h_\eps - \mathrm{i}\, \sum_{k=1}^2   \d_{x_k} (c_{k \eps} 
    \d_{x_k} h_\eps) = f_\eps - f_0 + 
    \mathrm{i}\, \sum_{k=1}^2   \d_{x_k} ((c_{k \eps} - c_{k0}) 
    \d_{x_k} v) =: \widetilde{f}_\eps, \quad
    h_\eps \mid_{t=0} = g_\eps - g_0,
$$
which imply an energy estimate of the form \eqref{energy} for $h_\eps$ and with suitably adapted initial value and right-hand side $\widetilde{f}_\eps$ instead. 
As noted in the proof of Corollary \ref{cor} we have appropriate norm convergence of the standard regularizations $g_\eps$, $f_\eps$ of the data $g_0$,$f_0$. In addition, we now have to call on uniform boundedness of $\norm{\d_t c_{\eps}}{L^\infty(\R^2\times[0,T])}$ and on the $\Con^1([0,T],H^{r+1}(\R^2))$-norm convergence $c_{k \eps} - c_{k0} \to 0$ ($\eps \to 0$). Using the continuity of $H^{r+1}(\R^2) \cdot H^2(\R^2) \to H^{r+1}(\R^2)$ we obtain $\widetilde{f}_\eps \to 0$ in $\Con^1([0,T],H^{r+1}(\R^2))$. Thus we finally obtain the convergence $u_\eps \to v$  in $\Con([0,T],H^1(\R^2))$, and in particular we have $u \approx v$. 
\end{remark}


\begin{example} Consider the strictly positive square root of $\de$  represented by $(\sqrt{\rho_\eps})_{\eps\in\,]0,1]}$, where $\rho$ is a mollifier as in Proposition \ref{Wurzel}. (Note that we obtain $h_\eps = \de * \rho_\eps = \rho_\eps$ in this case.) To simplify technical matters, assume in addition that $\rho$ satisfies the conditions discussed in Remark \ref{Wurzel_rem}(ii) and moreover $\sqrt{\rho} \in L^1(\R^n)$. (For example, any suitably normalized function of the form $x \mapsto (1+|x|^2)^{-m/2}$ with $m > 2n$ would do.) Then we may consider the class $g \in \G_{L^2}(\R^n)$, given directly by $(\sqrt{\rho_\eps})_{\eps\in\,]0,1]}$, as a square root of $\de$. 

We have $g^2 \approx \de$, but $g \approx 0$, which is easily seen by action on a test function $\vphi$ upon substituting $y = x/\eps$ in $\int \!\sqrt{\rho_\eps(x)}\, \vphi(x)dx = \eps^{n/2} \int\! \sqrt{\rho(y)}\, \vphi(\eps y)dy$ and applying dominated convergence (thereby using that $\sqrt{\rho} \in L^1$). In essence, this effect has already been observed earlier in the generalized function model of  ultrarelativistic Reissner-Nordstr{\o}m fields in \cite[Equations (15) and (17)]{Steinbauer:97}.

The Cauchy problem (\ref{PDE}-\ref{IC}) with generalized initial value $g$, right-hand side $f=0$, constant coefficients $c_k = 1$ ($k=1,\ldots,n$) 
and $V = 0$, written out for representatives then reads 
$$
    \d_t u_\eps = i \Delta u_\eps, 
      \quad u_\eps|_{t=0} = \sqrt{\rho_\eps}.
$$
The solution is given by the action of the strongly continuous unitary group $U_t := \exp(i t \Delta)$ ($t \in \R$) of operators on $L^2(\R^n)$ in the form $u_\eps(t,x) = (U_t \sqrt{\rho_\eps})(x)$. 

Let $t \in \R$ and $\mu^t_\eps$ denote the positive measure on $\R^n$ with density function $|u_\eps(t,.)|^2$ for the Lebesgue measure. By unitarity of $U_t$ we obtain
$$
  \mu^t_\eps(\R^n) = \int_{\R^n} |u_\eps(t,x)|^2\, dx =
   \int_{\R^n} (U_t \sqrt{\rho_\eps})(x) \cdot
     \ovl{(U_t \sqrt{\rho_\eps})(x)}\, dx = 
     \int_{\R^n} |\sqrt{\rho_\eps(x)}|^2\, dx
     = \int_{\R^n} \rho_\eps(x)\, dx = 1, 
$$
hence $\{\mu^t_\eps : t \in \R, \eps \in \, ]0,1]\}$ is a family probability measures on $\R^n$ and $\norm{\mu^t_\eps}{\text{var}} = 1$ ($t \in \R$, $\eps \in \, ]0,1]$) holds in the Banach space of finite measures.  

We claim that for any $t \neq 0$ the net $(\mu^t_\eps)_{\eps \in \, ]0,1]}$  converges to $0$ with respect to the vague topology on finite measures (cf.\ \cite[\pg 30]{Bauer:01}): 
 Since $\sqrt{\rho_\eps} \in L^1(\R^n)$ we obtain from the $L^1$-$L^\infty$-estimate for the Schr\"odinger propagator (\cite[\pg 4.4, Theorem 1]{Rauch:91}) that $\norm{u_\eps(t,.)}{L^\infty} \leq \norm{\sqrt{\rho_\eps}}{L^1} / (4 \pi |t|)^{n/2}$ and therefore
for any $\psi \in \Conc(\R^n)$
\begin{multline*}
  |\dis{\mu^t_\eps}{\psi}| \leq 
    \int_{\R^n} |u_\eps(t,x)|^2| \, |\psi(x)|\, dx
    \leq (4 \pi |t|)^{-n}\, \norm{\psi}{L^1} 
       \norm{\sqrt{\rho_\eps}}{L^1}\\
    =  (4 \pi |t|)^{-n}\, \norm{\psi}{L^1} \norm{\sqrt{\rho}}{L^1}\,  
      \eps^{n/2} \to 0 \qquad (\eps \to 0).
\end{multline*} 
The same obviously holds with $\psi \in \S(\R^n)$, hence also $\mu^t_\eps \to 0$ in $\S'(\R^n)$ (when $t \neq 0$), whereas $(\mu^t_\eps)_{\eps \in \, ]0,1]}$ can certainly not be weakly convergent as net of finite measures (i.e., in the weak-$*$ topology in the dual of the space of bounded continuous functions on $\R^n$) since $\dis{\mu^t_\eps}{1} = \mu^t_\eps(\R^n) = 1 \not\to 0$ as $\eps \to 0$ (cf.\ also \cite[Theorem 30.8]{Bauer:01}).

\end{example}




\paragraph{\emph{Acknowledgement:}} The author wishes to express his sincere thanks to two anonymous referees. In particular, we thank the one for hints to additional references of related work, and we are deeply grateful to the second referee for patient comments and clarifications that lead to crucial corrections in Theorem 3.1 and its proof.

\bibliography{gue}
\bibliographystyle{abbrv}

\end{document}